\documentclass{amsart}
\usepackage{amssymb}
\usepackage{amsmath,amssymb, amsthm, fourier,mathrsfs}
\usepackage{enumitem,dsfont}
\usepackage{graphicx}
\usepackage{cite}
\usepackage{xcolor}

\usepackage{slashed}
\usepackage{mathtools}
\mathtoolsset{showonlyrefs}

\newtheorem{Theorem}{Theorem}[section]

\newtheorem{Lemma}[Theorem]{Lemma}

\newtheorem{Remark}{Remark}[section]

\title{Sobolev spaces for singular  perturbation of Laplace operator}

\author[V. Georgiev]{Vladimir Georgiev}

\address[V.Georgiev]{
Department of Mathematics,
University of Pisa,
Largo Bruno Pontecorvo 5,
I - 56127 Pisa, Italy}
\address{
Faculty of Science and Engineering, Waseda University,
3-4-1, Okubo, Shinjuku-ku, Tokyo 169-8555, Japan}
\address{
Institute of Mathematics and Informatics,  Bulgarian Academy of Sciences, Acad. G. Bonchev Str., Block 8, Sofia, 1113, Bulgaria
}
\email{georgiev@dm.unipi.it}
\author[M.Rastrelli]{Mario Rastrelli}
\address[M.Rastrelli]{
Department of Mathematics,
University of Pisa,
Largo Bruno Pontecorvo 5,
I - 56127 Pisa, Italy}

\thanks{
 V.G. and M.R. were partially supported by  GNAMPA 2023 and  by the project PRIN  2020XB3EFL by the Italian Ministry of Universities and Research.
V.G. was partially supported by    the Top Global University Project, Waseda University, by the University of Pisa, Project PRA 2022 85 and by Institute of Mathematics and Informatics, Bulgarian Academy of Sciences.
}


\begin{document}

\begin{abstract}
We study the perturbed Sobolev space $H^{1,r}_\alpha$, $r \in (1,\infty),$ associated with singular perturbation $\Delta_\alpha$ of Laplace operator in Euclidean space of dimension $2.$ The main results give the possibility to extend the $L^2$ theory of perturbed Sobolev space to  the $L^r$ case. When $r \in (2,\infty)$ we have appropriate representation of the functions in $H^{1,r}_\alpha$ in regular and singular part. An application to local well - posedness of the NLS associated with this singular perturbation in the mass critical and mass supercritical cases is established too.
\end{abstract}

\maketitle

\section{Introduction and basic definitions}

We study the singular-perturbed Laplacian  $-\Delta_{\alpha}$, $\alpha\in \mathbb{R}$, on $L^2(\mathbb{R}^2)$, that is a delta-like perturbation of the free Laplacian in $\mathbb{R}^2$.
This operator describes a zero-range interaction between particles or the presence of an impurity.
The parameter $\alpha$ expresses, in suitable units, the inverse scattering length of the interaction supported at $x_0=0$ and for   $\alpha=\infty$ it is the classical Laplacian on $L^2.$
 We mention that $-\Delta_{\alpha}$ is a non-trivial self-adjoint extension on $L^2$ of the symmetric operator
 $$-\Delta|_{C^\infty_0(\mathbb{R}^2\setminus\{0\})}.$$
This is today a well-known  one parameter class of operators, since the first rigorous attempt \cite{BF61} by Berezin and Faddeev in 1961, the seminal work \cite{AH81} by Albeverio and H\o{}egh-Krohn in 1981, and subsequent characterizations by many others authors.
In dimension $d=2,3$,  the explicit characterization of the domain $\mathcal{D}(-\Delta_\alpha)$ is well-known and consists of functions that can be decomposed in a $H^2$ function plus a singular one, the Green function. The similarities with the Sobolev spaces gave the name of the perturbed Sobolev space $H^2_\alpha=\mathcal{D}(-\Delta_\alpha)$. In \cite{GMS18} there is a focus on the operator $(-\Delta_\alpha)^{s/2},$ with $s\in (0,2),$  that provides the definition of the singular-perturbed Sobolev Spaces $H^s_\alpha=\mathcal{D}((-\Delta_\alpha)^{s/2})$. These spaces allow the study of PDEs like  nonlinear Schr\"odinger equations (NLSE) with a point interaction, thanks to the develop of tools like Strichartz estimates \cite{DMSY18,IS17}.

 One of the basic motivation of our work is to develop some missing tools of harmonic analysis needed to study dispersive equations associated with $-\Delta_{\alpha}$.

For example, the classical Nonlinear Schr\"odinger equation (NLSE)
\begin{equation}\label{eq.NLSEfr}
    ( i \partial_t + \Delta) u =\mu u|u|^{p-1}, \ p >1, \mu = \pm 1,
\end{equation}
is a well studied object and we can immediately refer to the book of T. Cazenave \cite{c} and list the following basic tools used to obtain local and global well - posedness for \eqref{eq.NLSEfr}:
\begin{enumerate}
    \item  Strichartz estimates for the linear Schr\"odinger group
$e^{it \Delta}$ using $L^q(0,T)L^r(\mathbb{R}^2)$ spaces with admissible couples $(q,r);$

    \item systematic use of the Sobolev spaces $W^{1,p}(\mathbb{R}^2) = H^{1,p}(\mathbb{R}^2)  $ and application of Strichartz estimates using these spaces. As a typical example we can recall the following estimate of the composition operator
    \begin{equation}\label{eq.co100}
        \|u|u|^{p-1}\|_{H^{1,r}(\mathbb{R}^2)} \lesssim \|u\|^p_{H^1(\mathbb{R}^2)}
    \end{equation}
for $r<2$ and $r$ sufficiently close to $2.$
\end{enumerate}

While Strichartz estimates for the  Schr\"odinger group
$e^{it \Delta_\alpha}$ are obtained recently by establishing the existence and completeness of the wave operators in $L^r$( see \cite{DMSY18}, \cite{CMY19}, \cite{CMY19b}), to our knowledge there is a lack  of results on the definition and properties of Sobolev spaces $H^{1,r}_\alpha(\mathbb{R}^2)$ associated with the perturbed operator $\Delta_\alpha.$  Let us mention that the necessity to have $L^p$ version of classical Sobolev spaces $H^{1,p}$ is dicussed  in \cite{GMS22}, \cite{GM22Ob}, \cite{FGI22}.

Our Theorems \ref{t.2.1} and \ref{t.2.2} will show a new description of $H^{1,p}_\alpha$ in dimension $d=2$, moving the focus also into $L^p$ spaces.  This is crucial because it allows to gain new Strichartz estimates that involve the spaces $H^{1,p}_\alpha$ and the energy space $H^1_\alpha$. Now the contraction argument is available and we can give a new proof of local well - posedness of the following NLSE:
\begin{equation}\label{eq.NLSE}
    ( i \partial_t + \Delta_\alpha) u =\mu u|u|^{p-1}, \ p >1, \mu = \pm 1.
\end{equation}

The key estimate we use in this local existence result is the following variant of \eqref{eq.co100}
\begin{equation}\label{eq.co119}
        \|u|u|^{p-1}\|_{H^{1,r}_\alpha(\mathbb{R}^2)} \lesssim \|u\|^p_{H^1_\alpha(\mathbb{R}^2)}
    \end{equation}
      for $r<2$ and $r$ sufficiently close to $2.$


\subsection{Overview on existing results}

First we give a short description on results treating the Strichartz estimates.
 for $-\Delta_\alpha$ in dimension $d=2.$

In the case of  dimension $d=2$ and when we have  $N$ singularities, the existence and completeness on $L^2(\mathbb R^2)$ of wave operators $W_{\alpha,N}^\pm$ is well-known thanks to Kato-Birman-Rosenblum theorem \cite{KATO66}. Using this result, Cornean, Michelangeli and Yajima in \cite{CMY19,CMY19b} defined a condition of regularity  for the singularities and proved that, under this condition (always fulfilled for $N=1$, that is our case), the wave operators are bounded in $L^p$ for every $p$. As a corollary they obtained the $L^{p^\prime}-L^{p}$ estimates without weights for every $p\in [2,\infty)$ and, immediately after, the Strichartz estimates:

\begin{equation}
    \begin{aligned}
& \left\| e^{i t \Delta_\alpha} P_{ac}f \right\|_{L^q(\mathbb{R}_t)L^p(\mathbb{R}^2)} \lesssim \|f\|_{L^2}, \\
& \left\|\int_0^t e^{i(t-\tau)\Delta_\alpha} P_{ac}F(\tau) d\tau\right\|_{L^q(\mathbb{R}_t)L^p(\mathbb{R}^2)} \lesssim  \left\| F\right\|_{L^{{s}'}(\mathbb{R}_t)L^{{r}'}(\mathbb{R}^2) },
    \end{aligned}
\end{equation}
where the couples $(p,q)$ and $(s,r)$ are 2-dimensional Strichartz exponents, i.e.
$$\frac{1}{p}+\frac{1}{q}=\frac{1}{2},\ \ 2<q\leq \infty.$$
For $N=1$ and the explicit structure of the absolutely continuous subspace for $-\Delta_\alpha$ allows to generalize the above inequalities local in time, without the orthogonal projection.
These results, with the perturbed Sobolev spaces develop, led to an in intense research in PDEs.

For the case $d=3$ we refer to \cite{DPT06} and \cite{IS17}.
In \cite{DMSY18} Dell'Antonio, Michelangeli, Scandone and Yajima proved that the wave operators associated to the pair $(-\Delta_\alpha,-\Delta)$, defined by the following strong limits,
$$W_{\alpha,N}^\pm=\lim_{t\to \pm\infty}e^{it\Delta_{\alpha,N}}e^{-it\Delta}$$
exist, they are complete in $L^2(\mathbb{R}^3)$ and are bounded on $L^q(\mathbb{R}^3)$ for $1<q<3$.

As a consequence,  the weighted estimate
\begin{equation}\label{eq.APT}
    \|w^{-1}e^{it\Delta_{\alpha,N}}P_{ac}f\|_{L^\infty(\mathbb{R}^3)}\leq {C}{t^{-\frac{3}{2}}}\|wf\|_{L^1(\mathbb{R}^3)}
\end{equation}
from \cite{DPT06}
can be extended in the $L^p-L^q$ setting with
\begin{equation}\label{eq.IS}
      \|e^{it\Delta_{\alpha,N}}P_{ac}f\|_{L^{p}(\mathbb{R}^3)}\leq {C}{t^{-\frac{3}{2}\left(\frac{1}{p^\prime}-\frac{1}{p}\right)}}\|f\|_{L^{p^\prime}(\mathbb{R}^3)}\ \ \mbox{for} \ \  p\in[2,3).
\end{equation}

The local well - posedness is a necessary element in the study of standing waves. Let us give a brief overview on the results in this direction.

The existence of standing waves for the  2d Hartree type equation with point interaction
is  studied in \cite{GMS22b}
\begin{equation}
    i\partial_tu=-\Delta_\alpha u +(w*|u|^2)u,
\end{equation}
where $w$ is a real-valued measurable function.

The existence and stability of standing waves for power nonlinearity is studied in \cite{FGI22} in all possible cases: mass critical, mass subcritical and subcritical ones. In alternative way in the mass subcritical case the ground states can be obtained by looking for constraint minimization of the energy. This approach is developed in \cite{ABCT}. The 3d case can be seen in \cite{ABCT22}.

Each of the works treating existence and stability/instability of standing waves need appropriate local well posedness in energy space. In all of them this is dome by appropriate modification of Cazenave approach by using a compactness argument.

It is not difficult (using only Strichartz estimates to establish
  local existence and uniqueness of a solution in $C([0,T]; L^2)$ in the mass sub-critical case (see  in Theorem \ref{t.le19} below).

  The main novelty is the prove of  the local existence and uniqueness of the solution of \eqref{eq.NLSE} in $C([0,T]; H^1_\alpha)$, for  the mass critical and super critical cases $(p\geq 3)$,  by using a classical contraction argument. The precise statement is given  in Theorem \ref{t.lems8}.

\subsection{Heuristic introduction of $\Delta_\alpha$}
An important feature of the family $-\Delta_{\alpha}$, $\alpha\in \mathbb{R}$, is the following explicit formula for the resolvent, valid for sufficiently large $\omega > 0$.
\begin{equation}\label{eq:res_formula}
(-\Delta_\alpha+\omega)^{-1} f\;=\;(-\Delta+\omega)^{-1}f + \frac{1}{\beta_{\alpha}(\omega) }\mathbb{G}_\omega \langle f, \mathbb{G}_\omega \rangle .
\end{equation}
Here $\beta_\alpha(\omega) = \alpha + c(\omega),$ where $c(\omega)$ is associated with the asymptotics of $\mathbb{G}_\omega(x) $ near the origin.
Identity \eqref{eq:res_formula} says that the resolvent of $-\Delta_\alpha$ is a rank-one perturbation of the free resolvent.

Formal substitution $$ \phi =(-\Delta_\alpha+\omega)^{-1} f,\ \  g= (-\Delta+\omega)^{-1}f$$ gives
\begin{equation}\label{eq.dom96}
    \phi = g + \frac{1}{\beta_{\alpha}(\omega) }\mathbb{G}_\omega \langle f, \mathbb{G}_\omega \rangle =
    g+ \frac{1}{\beta_{\alpha}(\omega) }\mathbb{G}_\omega \langle g, (\omega-\Delta) \mathbb{G}_\omega \rangle.
\end{equation}

Our  choice of the singular perturbation is determined by the requirement
\begin{equation}\label{eq.gom01}
    (\omega-\Delta) \mathbb{G}_\omega  = \delta.
\end{equation}
Hence,
    \begin{equation}\label{eq:defGlambda}
\mathbb{G}_\omega(x)\;= (2\pi)^{-1} K_{0}(\sqrt{\omega} |x|),
\end{equation}
where $K_{0}$ is modified Bessel function of order zero. The constant $\beta_{\alpha}(\omega) $ in \eqref{eq:res_formula} shall be determined by the special asymptotics of $K_0$ near the origin. We shall turn to this choice later on.

Therefore if we assume formally that $\omega - \Delta_\alpha$ is a self -adjoint positive operator (on $L^2$ for example and its resolvent is rank one perturbation of type \eqref{eq:res_formula} with $\mathbb{G}_\omega$ determined by \eqref{eq.gom01}, then we have
\begin{equation}\label{eq.dom107}
  \mathcal{D}(-\Delta_{\alpha}) = \left\{ \phi \in L^2;    \phi  =    g+ \frac{1}{\beta_{\alpha}(\omega) } g(0) \mathbb{G}_\omega, g \in H^2 \right\}.
\end{equation}

The space in \eqref{eq.dom107} shall be denoted by $H^2_\alpha(\mathbb{R}^2).$ Since this is a linear space we can write
\begin{equation}\label{eq.dom116}
 H^2_\alpha(\mathbb{R}^2) =  \mathcal{D}(-\Delta_{\alpha}) = \left\{ \phi \in L^2;    \phi  =    \beta_{\alpha}(\omega) g +  g(0) \mathbb{G}_\omega, g \in H^2 \right\}.
\end{equation}

Taking $g\in H^2$ such that $g(0)=1,$ we see that $G_\omega$ is in the domain of $\Delta_\alpha$ if and only if
\begin{equation}
    \beta_\alpha(\omega) = \alpha + c(\omega) =0.
\end{equation}
In this case, using \eqref{eq.dom116} and taking
$$ \phi  =    \beta_{\alpha}(\omega) g +  g(0) \mathbb{G}_\omega = \mathbb{G}_\omega.$$
we see that
$$ (\omega-\Delta_\alpha) \mathbb{G}_\omega =0.$$

To close our heuristic description  of $\Delta_\alpha$ let us note that \eqref{eq.dom96} implies
$$ (\omega- \Delta_\alpha) \phi = f = (\omega-\Delta)g.$$

\subsection{Precise definitions}
The above observation enables one to construct a rigorous definition of $\mathcal{D}(-\Delta_\alpha)$
and its action. Namely,
\begin{gather}
\label{eq:op_dom}\mathcal{D}(-\Delta_\alpha)\;=\;\Big\{\phi\in L^2(\mathbb{R}^2)\,\Big|\,\phi=g_{\alpha,\omega}+\frac{ g_{\alpha,\omega}(0)}{\beta_\alpha (\omega)}\,\mathbb{G}_\omega\textrm{ with }g_{\alpha,\omega}\in H^2(\mathbb{R}^2)\Big\},\\
\label{eq:opaction}
 (-\Delta_\alpha+\omega)\,\phi\;=\;(-\Delta+\omega)\,g_\omega\,,
\end{gather}
where $\omega>0$ is an arbitrarily fixed constant,
\begin{equation}\label{eq.2d1}
    \beta_\alpha(\omega):= \alpha + c(\omega), \ c(\omega)= \frac{\gamma}{2\pi} + \frac{1}{2\pi} \ln \left( \frac{\sqrt{\omega}}{2} \right),
\end{equation}
$\gamma$ denoting Euler-Mascheroni constant,
Owing to \eqref{eq:defGlambda}, we have the explicit formula
\begin{equation}\label{eq:defGlambda47}
\mathbb{G}_\omega(x)\;= (2\pi)^{-1} K_{0}(\sqrt{\omega} |x|).
\end{equation}

The results in \cite{AH81} guarantee the following statements:
\begin{itemize}
  \item we have the relation
  \begin{equation}\label{eq.h26}
      \mathcal{D}(-\Delta_\alpha) = H^2_\alpha(\mathbb{R}^2) = \left\{ \phi \in L^2; \phi = (\omega-\Delta_\alpha)^{-1}f, f \in L^2   \right\};
  \end{equation}
    \item the operator $\Delta_\alpha$ is self-adjoint, its spectrum consist of absolutely continuous part $(-\infty,0]$ and it has point eigenvalue at $\omega_0$ determined by $\alpha + c(\omega_0)=0,$ i.e.
    \begin{equation}\label{eq.pi1}
        \omega_0= 4 e^{-4\pi\alpha - 2\gamma};
    \end{equation}
        \item the domain and the action are independent of the choice of $\omega;$
    \item the resolvent identity \eqref{eq:res_formula} holds.
\end{itemize}
Summarizing, we can define the perturbed Sobolev space $H^2_\alpha(\mathbb{R}^2)=\mathcal{D}(-\Delta_\alpha)$ and moreover $(\omega-\Delta_\alpha)^{1/2}$ is well defined. In particular, as in the case of classical Sobolev spaces, we can define $H^1_\alpha(\mathbb{R}^2)=\mathcal{D}((\omega-\Delta_\alpha)^{1/2})$.
The space $H^1_\alpha(\mathbb{R}^2) $ is defined explicitly  in \cite{MOS-2018FP} as  follows
\begin{equation}\label{eq.defh1a81}
   H^1_\alpha(\mathbb{R}^2) = \left\{\phi = g + c \mathbb{G}_\omega, g \in H^1, c \in \mathbb{C}\right\}
\end{equation}
and the corresponding norm is
\begin{equation}\label{eq.ns77}
    \|\phi\|_{H^1_\alpha}^2 = \|g\|_{H^1}^2 + |c|^2.
\end{equation}

\subsection{$L^p$ extension of the resolvent}

First we define $H^{2,p}_\alpha$ by relation similar to \eqref{eq:op_dom}.

\begin{gather}
\label{eq:op_dom63}H^{2,p}_\alpha\;=\;\Big\{\phi\in L^p(\mathbb{R}^2)\,\Big|\,\phi=g_{\alpha,\omega}+\frac{ g_{\alpha,\omega}(0)}{\beta_\alpha (\omega)}\,\mathbb{G}_\omega\textrm{ with }g_{\alpha,\omega}\in H^{2,p}(\mathbb{R}^2)\Big\},\\
\label{eq:opaction67}
 (-\Delta_\alpha+\omega)\,\phi\;=\;(-\Delta+\omega)\,g_\omega\,,
\end{gather}
where $\omega > \omega_0 = 4 e^{-4\pi\alpha - 2\gamma} $ with $\omega_0$ being the unique eigenvalue of $\Delta_\alpha.$

Using the resolvent relation \eqref{eq:res_formula} one can easily obtain the following.

\begin{Lemma}
    If $\omega > \omega_0 = 4 e^{-4\pi\alpha - 2\gamma}, $
    the operator
    $$ (\omega - \Delta_\alpha)^{-1}$$ can be extended as a closed operator on $L^p.$
\end{Lemma}

Next, we define the spaces $H^{1,p}_\alpha$ for  $p \in (1,\infty).$

\begin{equation}\label{d}
  H^{1,p}_\alpha(\mathbb{R}^2) = \left\{\phi \in L^p; \phi = (\omega - \Delta_\alpha)^{-1/2}f, f \in L^p \right\} .
\end{equation}

It is not difficult to check the following
\begin{Lemma}\label{l.sem16}
   For any $q \in (2,\infty)$ there the space $H^2_\alpha$ is dense in $H^{1,q}.$
\end{Lemma}

\begin{proof}
It is sufficient to use the definition of perturbed Sobolev space, the Sobolev embedding of Lemma \ref{l.sem1} so we can conclude that $H^2_\alpha$ is embedded in $H^{1,q}_\alpha.$ The density property follows from the fact that any $f \in L^q$ can be approximated by
$$ f_\varepsilon = (\omega-\varepsilon \Delta_\alpha)^{-1} f$$ as $\varepsilon \to 0.$ Therefore the Yosida approximation completes the proof.
\end{proof}

The key questions we discuss are:

\begin{enumerate}
    \item The space $H^2_\alpha$ is dense in $H^{1,p}_\alpha$ for  $p \in (1,\infty)$?

    \item Can we extend the characterization of $H^{1,p}_\alpha$  in a way similar to the relation
    $$ \phi=g+\frac{ g(0)}{\beta_\alpha (\omega)}\,\mathbb{G}_\omega$$
        used in \eqref{eq:op_dom63}, when  $p>2$?

        \item Can we say that $H^{1,p}$ and $H^{1,p}_\alpha$ coincide for $p \in (1,2)$ ?
\end{enumerate}

\section{Main results}

Our main result is quite similar to the $L^2$ case studied in \cite{GMS18}.

First of all we have the representation of the operator $ (\omega - \Delta_\alpha)^{-1/2}$ used in the definition of $H^{1,p}_\alpha.$
\begin{equation} \label{eq.Darr}
    \begin{aligned}
         (\omega - \Delta_\alpha)^{-1/2}f =  (\omega - \Delta)^{-1/2}f + \frac{1}{\pi}\int_0^\infty t^{-1/2} \mathbb{G}_{\omega+t}(|x|) \langle \mathbb{G}_{\omega+t}, f \rangle \frac{dt}{\beta_\alpha(\omega+t)}
    \end{aligned}
\end{equation}

To justify it we use the relation (4.7)  in \cite{GMS18} and we can write
\begin{equation}
    (\omega-\Delta_\alpha)^{-1/2} = \frac{1}{\pi}\int_0^\infty t^{-1/2} (\omega+t-\Delta_\alpha)^{-1} dt
\end{equation}
Using the relation \eqref{eq:res_formula} we get \eqref{eq.Darr}.

Our main canonical representation of this operator is closely connected with the following asymptotic representation of $\mathbb{G}_\omega(r)$.
\begin{equation}\label{eq.decomposition}
    \mathbb{G}_\omega(r) = \varphi_0(\sqrt{\omega}r) + R(\sqrt{\omega}r),
\end{equation}
where
\begin{equation}\label{eq.vp08}
    \varphi_0(r) =  - (2\pi)^{-1}\left( \log \left( r/2\right) + \gamma  \right) \phi(r),
\end{equation}
$\phi$ is a smooth non - negative function, such that
\begin{equation}
    \phi(r) =
  \left\{  \begin{aligned}
       & 1, \ \ \mbox{if} \ \ r<1;\\
       & 0, \ \ \mbox{if} \ \ r>2.
    \end{aligned}\right.
\end{equation}

Our first result is the following.

\begin{Theorem} \label{t.2.1}
  We have the following properties:
  \begin{itemize}
      \item if $p>2$ and $f \in L^p,$ then there exists a unique $g \in H^{1,p}(\mathbb{R}^2)$ so that
      \begin{equation}
          (\omega - \Delta_\alpha)^{-1/2}f = g + \varphi_0(\sqrt{\omega}r) C(f),
      \end{equation}
      where $C(f)$ is the linear functional defined by
      \begin{equation}
       C(f) =   \frac{1}{\pi}\int_0^\infty t^{-1/2}  \langle \mathbb{G}_{\omega+t}, f \rangle \frac{dt}{\beta_\alpha(\omega+t)}
      \end{equation}
      that is bounded functional on $L^p:$
      \item if $p \in (1,2),$ then $$ (\omega - \Delta_\alpha)^{-1/2}f  \in H^{1,p}(\mathbb{R}^2)$$
      and we have the estimate
      \begin{equation}\label{eq.stimaellittica}
          \|(\omega-\Delta)^{1/2} (\omega-\Delta_\alpha)^{-1/2}f\|_{L^p} \lesssim \|f\|_{L^p}.
      \end{equation}
  \end{itemize}
\end{Theorem}

By using the approximation
$$  (\omega-\Delta_\alpha)^{-1/2}f = \lim_{\varepsilon \to 0}   (\omega-\Delta_\alpha)^{-1/2} f_\varepsilon, \ \ f_\varepsilon = \omega^{1/2} (\omega-\varepsilon \Delta_\alpha)^{-1/2}f $$
we obtain the following.

\begin{Theorem}\label{t.2.2}
 If $p>2$ and $f \in L^p,$ then there exists a unique $g \in H^{1,p}(\mathbb{R}^2)$ so that
      \begin{equation}\label{eq.rappresentazione}
          (\omega - \Delta_\alpha)^{-1/2}f =g+\frac{ g(0)}{\beta_\alpha (\omega)}\,\mathbb{G}_\omega.
      \end{equation}
\end{Theorem}

Next we turn to an application of the above results.
We consider the following NLS associated with $\Delta_\alpha$
\begin{equation}\label{eq.CP81}
\begin{aligned}
    &( i \partial_t + \Delta_\alpha) u =\mu u|u|^{p-1}, \ p >1, \mu = \pm 1.\\
   & u(0) = u_0 \in H^1_\alpha(\mathbb{R}^2).
    \end{aligned}
\end{equation}
Formally, the conservation of mass and energy are associated with the relations
\begin{equation}
     \|u(t)\|^2_{L^{2}(\mathbb{R}^2)}  =  \|u(0)\|^2_{L^{2}(\mathbb{R}^2)}  , \ \ E(t) = E(0),
\end{equation}
where
\begin{equation}
\begin{aligned}
    & E(t) =   \frac{1}{2} \langle -\Delta_\alpha u(t), u(t) \rangle_{L^2} + \frac{\mu}{p+1} \|u(t)\|^{p+1}_{L^{p+1}(\mathbb{R}^2)} = \\
    & =\frac{1}{2} \left\| (\omega-\Delta_\alpha)^{1/2} u(t)\right\|^2_{L^2} -  \frac{\omega}{2} \|u(t)\|^2_{L^{2}(\mathbb{R}^2)}   + \frac{\mu}{p+1} \|u(t)\|^{p+1}_{L^{p+1}(\mathbb{R}^2)}.
\end{aligned}
\end{equation}
First we consider the mass subcritical case $p \in (1,3)$ and we state the following local existence result in $L^2.$
\begin{Theorem}\label{t.le19}
    For any $p \in (1,3)$ and any $R>0$ there  exists $T=T(R,p)>0$ so that for any
    $$ u_0 \in B_{L^2}(R) = \left\{ \phi \in L^2;  \|\phi\|_{L^2} \leq R \right\}$$
    there exists a unique solution
    $$ u \in C([0,T]; L^2)$$
    to the integral equation
\begin{equation} \label{eq.ie30}
    u = e^{it \Delta_\alpha} u_0 -i \int_0^t e^{i(t-\tau)\Delta_\alpha}  u(\tau)|u(\tau)|^{p-1} d\tau.
\end{equation}
    associated to \eqref{eq.CP81}.
\end{Theorem}

\begin{Remark}
    Using a rescaling argument (see Section \ref{sec.as8}) in the mass subcritical case and assuming initial data in $H^1_\alpha$
    one can prove the conservation of energy and global existence result.
\end{Remark}
\begin{Remark}
    In the mass critical $(p=3)$ and mass super critical $(p>3)$ cases we can obtain local existence result in $H^1_\alpha.$ See Theorem below
\end{Remark}

\section{Characterization of $H^{1,p}_\alpha$}

We start with the well-known fact that $H^{1,p}(\mathbb{R}^2)$ has as a  norm (see section 1.3.1 and the identity i) in section 1.4.1 in \cite{YSY10})
\begin{equation}\label{eq.s1p3}
   \|u\|_{H^{1,p}} =  \|u\|_{W^{1,p}} = \sum_{|\alpha|\leq 1}\|\partial_x^\alpha u\|_{L^p} , \ \ 1 < p < \infty.
\end{equation}
Another equivalent norm is the following one
\begin{equation}
   \|u\|_{H^{1,p}} =  \|(1-\Delta)^{1/2}u\|_{L^p} ,  \ \ 1 < p < \infty.
\end{equation}

\begin{Remark}
    Using the norm \eqref{eq.s1p3},   we obtain  the following property: the functions $\varphi_0$ defined in \eqref{eq.vp08}
    as well as the function $\mathbb{G}_\omega$ defined in \eqref{eq.decomposition}
    are in $H^{1,p}(\mathbb{R}^2)$ if and only if $p<2.$
\end{Remark}

\begin{proof}[Proof of Theorem \ref{t.2.1}]
Using the rescaling argument of Section \ref{sec.res2}, we can assume $\omega =1$ and $\alpha$ is so large that
the unique eigenvalue $4 e^{-4\pi\alpha - 2\gamma}$ determined in \eqref{eq.pi1} is in the interval $(0,1).$

    We start from the case $p>2$. Thanks to the decomposition of $\mathbb{G}_{1+t}$ in \eqref{eq.decomposition}, we can write
    \begin{equation} \label{eq.cpg160}
        (1-\Delta_\alpha)^{-1/2}f=g+   \frac{1}{\pi}\varphi_0(r)\int_0^\infty t^{-1/2}  \langle \mathbb{G}_{1+t}, f \rangle \frac{dt}{\beta_\alpha(1+t)},
    \end{equation}
with
\begin{equation}\label{eq.def g}
   g=(1 - \Delta)^{-1/2}f+\Gamma(f)+\Gamma_0(f),
\end{equation}
where
\begin{equation}\label{eq.Ga16}
\begin{aligned}
     & \Gamma (f) (x) = \int_0^\infty t^{-1/2} R( \sqrt{t+1} |x|)  \left(\int_{\mathbb{R}^2}   \mathbb{G}_{1+t}( |y|) \overline{f(y)}  dy\right) dt
\end{aligned}
\end{equation}
and
\begin{equation}\label{eq.Ga022}
\begin{aligned}
     & \Gamma_0 (f) (x) = \int_0^\infty t^{-1/2} [\varphi_0( \sqrt{t+1} |x|) - \varphi_0(|x|)] \left(\int_{\mathbb{R}^2}   \mathbb{G}_{1+t}( |y|) \overline{f(y)}  dy\right) dt.
\end{aligned}
\end{equation}
Thanks to Lemmas \ref{l.gamma} and \ref{l.gamma0}, in the  Section \ref{s.lp} we will prove that $\Gamma(f),\Gamma_1(f)\in H^{1,p}$.
We note in particular that $\varphi_0(r)\notin H^{1,p}$, for $p>2$ because $|\partial_x^\alpha \varphi_0(r)|\sim 1/r$ near zero.

If $1<p<2$, we have instead
$$(1-\Delta_\alpha)^{-1/2}f=(1 - \Delta)^{-1/2}f+\Gamma(f)+\Gamma_1(f),$$
with
\begin{equation}
\begin{aligned}
     & \Gamma_1 (f) (x) = \int_0^\infty t^{-1/2} \varphi_0( \sqrt{t+1} |x|) \left(\int_{\mathbb{R}^2}   \mathbb{G}_{1+t}( |y|) \overline{f(y)}  dy\right) dt.
\end{aligned}
\end{equation}
 Again in  Section \ref{s.lp}, Lemmas \ref{l.gamma} and \ref{l.gamma1}  will give that $\Gamma(f),\Gamma_1(f)\in H^{1,p}$ and moreover \eqref{eq.lll86} and \eqref{eq.gamma1.2}, with an elliptic estimate give
 \begin{equation}
     \|(1-\Delta_\alpha)^{-1/2}f\|_{H^{1,p}}\lesssim\|f\|_{L^p},
 \end{equation}
that is equivalent to \eqref{eq.stimaellittica}.
\end{proof}
\begin{proof}[Proof of Theorem \ref{t.2.2}]
    We start with  the uniqueness. It follows by contradiction argument. So let us assume that there exist $g_1$ , $g_2\in H^{1,p}$ such that
    $$(\omega - \Delta_\alpha)^{-1/2}f =g_1+\frac{ g_1(0)}{\beta_\alpha (\omega)}\,\mathbb{G}_\omega$$
    and
    $$(\omega - \Delta_\alpha)^{-1/2}f =g_2+\frac{ g_2(0)}{\beta_\alpha (\omega)}\,\mathbb{G}_\omega.$$
    From these two equations we obtain that
    $$g_1-g_2=\frac{ g_2(0)-g_1(0)}{\beta_\alpha (\omega)}\,\mathbb{G}_\omega.$$
    We remember that $\mathbb{G}_\omega\notin H^{1,p}$ for $p>2$, because its first derivative behaves like $1/r$ for $r$ near zero, so $ g_2(0)-g_1(0)=0$ and $g_1=g_2$.

    Our next step is to prove that for any $\phi =  (1-\Delta_\alpha)^{-1/2}f  \in H^{1,p}_\alpha$ we have
    \begin{equation}\label{eq.dec522}
        \phi=\Lambda(f)+C(f)\mathbb{G}_1,
    \end{equation}
    where
    $$ \Lambda: L^p \to H^{1,p}$$
    is a linear bounded operator and $C(f)$ is a bounded functional on $L^p.$

    As before we can assume $\omega=1$.
We have the relations \eqref{eq.def g} for $p >2.$ So we can write
 \begin{equation} \label{eq.cpg1515}
        (1-\Delta_\alpha)^{-1/2}f= \tilde{g} +   C(f)\varphi_0(r),
    \end{equation}
where
$$ \tilde{g}=(1 - \Delta)^{-1/2}f+\Gamma(f)+\Gamma_0(f). $$ Here $\Gamma$ and $\Gamma_0$  are $L^p-H^{1,p}$ continuous and $C(f)$ is a $L^p$ bounded functional. Recall that $G_1(r) = \varphi_0(r)+R(r) $ due to \eqref{eq.decomposition} and $R(|x|)$ is in $H^{1,p}.$ Hence defining
    $$g= \tilde{g} - C(f)R = (1 - \Delta)^{-1/2}f+\Gamma(f)+\Gamma_0(f)-C(f)R = \Lambda(f),$$
    we arrive at \eqref{eq.dec522}.

    Now we can  prove \eqref{eq.rappresentazione} by using \eqref{eq.dec522} and density argument.
   Let $\phi=(1-\Delta_\alpha)^{-1/2}f$  with $f \in L^q$ Then we can approximate $f$ by $f_n \in H^{1,p}_\alpha$ so that
    $f_n\to f $ in $L^p$. On one hand,
    $$ \phi_n = (1-\Delta_\alpha)^{-1/2}f_n \in H^2_\alpha $$
    so
    \begin{equation}
        \phi_n = g_n + g_n(0) \frac{\mathbb{G}_1}{\beta_1(\alpha)}
    \end{equation}
On the other hand , from \eqref{eq.dec522} we can deduce
\begin{equation}\label{eq.dec548}
        \phi_n=\Lambda(f_n)+C(f_n)\mathbb{G}_1,
    \end{equation}
    so the uniqueness observation discussed above implies
    $$ g_n = \Lambda(f_n), \ C(f_n) =   \frac{g_n(0)}{\beta_1(\alpha)} .$$
    and after taking the limit we get
    \begin{equation}
        \phi = g + g(0) \frac{\mathbb{G}_1}{\beta_1(\alpha)}.
    \end{equation}
This completes the proof.
\end{proof}

\section{$L^p$ estimates of the operators $\Gamma$} \label{s.lp}

The canonical representation \eqref{eq.Darr} shows that we have to consider (after rescaling) the term
\begin{equation}\label{eq.mop2}
    \frac{1}{\pi}\int_0^\infty t^{-1/2} \mathbb{G}_{1+t}(|x|) \langle \mathbb{G}_{1+t}, f \rangle \frac{dt}{\beta_\alpha(1+t)}.
\end{equation}

Then \eqref{eq:defGlambda} and asymptotics of Appendix \ref{sec.as8} imply that

\begin{equation}\label{eq:fgl10a}
\mathbb{G}_{1+t}(x)\;= \varphi_0( \sqrt{1+t}|x|) +R(\sqrt{1+t}|x|), \ \ t>0,
\end{equation}
where  $\varphi_0(r)$ is smooth in $(0,\infty)$ and satisfies
\begin{equation}
  \left\{  \begin{aligned}
       & \mathrm{supp} \  \varphi_0 (r) \subset \{ r \leq 2 \}, \\
       & \varphi_0(r)= - (2\pi)^{-1}\left( \log \left( r/2\right) + \gamma  \right) , \ \  r \leq 1 ,\\
       & | \partial_r^k \varphi_0(r)| \lesssim  r^{-k} \log^{1-k}(2/r)\  \ \ k=0,1,
    \end{aligned}\right.
\end{equation}
while the remainder $R$ is represented by two terms localised near $0$ and $\infty$ respectively. More precisely, we have
\begin{equation}
    R(r) = R_{small}(r) + R_{large}(r)
\end{equation}
where
\begin{equation}\label{eq.pss39}
  \left\{  \begin{aligned}
       & \mathrm{supp} \  R_{small}(r) \subset \{ r \leq 2 \}, \\
     & | \partial_r^k R_{small}(r)| \lesssim r^{2-k} \log(2/r) , r \leq 1, \ k=0,1, \\
    \end{aligned}\right.
\end{equation}
and
\begin{equation} \label{eq.pss47}
  \left\{  \begin{aligned}
       & \mathrm{supp} \  R_{large} (r) \subset \{ r \geq 1/\sqrt 2 \}, \\
       & R_{large}(r)\in L^q((0,\infty), rdr), \ \ \forall q \in (1,\infty), \\
     & \exists \delta>0, \ \  \mbox{so that} \  \ |R^\prime_{large}(\sigma)| \leq e^{-\delta \sigma} , \forall \sigma  \geq 1/\sqrt 2.
    \end{aligned}\right.
\end{equation}

We have the following

\begin{Lemma}
    If $\mathrm{arg} \lambda \in (\varepsilon, \pi)$ and $|\lambda| >0,$ then for any $p \in (1,\infty)$  we have
    \begin{equation}\label{eq.SSE9}
     \left\{   \begin{aligned}
       & \sum_{|\alpha|=m} \|\partial_x^\alpha R(\lambda |x|)\|_{H^{k,p}(\mathbb{R}^2_x)} \lesssim |\lambda|^{k+m-2/p}, k,m=0,1, \\
       & \|\varphi_0(\lambda |x|)\|_{L^p(\mathbb{R}^2_x)} \lesssim |\lambda|^{-2/p},\\
       &   \|G_{\lambda^2}(x)\|_{L^p(\mathbb{R}^2_x)} \lesssim |\lambda|^{-2/p}.
        \end{aligned}\right.
    \end{equation}
    Moreover for $p \in [1,2) $ we have
    \begin{equation}\label{eq.phi.2}
        \|\varphi_0(\lambda |x|)\|_{H^{1,p}(\mathbb{R}^2_x)} \lesssim |\lambda|^{1-2/p}.
    \end{equation}
\end{Lemma}
 The term \eqref{eq.mop2} suggests to  consider the operator
\begin{equation}
\begin{aligned}
     & \Gamma (f) (x) = \int_0^\infty t^{-1/2} R( \sqrt{t+1} |x|)  \left(\int_{\mathbb{R}^2}   \mathbb{G}_{1+t}( |y|) \overline{f(y)}  dy\right) dt.
\end{aligned}
\end{equation}
Note that for simplicity we do not put the factor $\beta_1(\alpha)$ in denominator, since this factor is bounded from below.

\begin{Lemma}\label{l.gamma}
    For any $p \in (1,\infty)$ the operator $\Gamma$ maps $L^{p}(\mathbb{R}^2)$ into $H^{1, p}(\mathbb{R}^2).$
\end{Lemma}

\begin{proof}
   It is sufficient to prove
\begin{equation}\label{eq.lll82}
    \|\Gamma(f)\|_{L^{q}} \lesssim \|f\|_{L^q}
\end{equation}
and
\begin{equation} \label{eq.lll86}
    \sum_{|\alpha|=1}\|\partial_x^\alpha \Gamma(f)\|_{L^{q,\infty}} \lesssim \|f\|_{L^q}
\end{equation}
and then apply Marcinkiewicz  interpolation theorem.

   Using \eqref{eq.SSE9} we have
   \begin{equation} \label{eq.11192}
   \begin{aligned}
     & |\Gamma(f)(x)|  \lesssim  \int_0^\infty t^{-1/2} R( \sqrt{t+1} |x|)  \left\|  \mathbb{G}_{1+t}(  |y|) \right\|_{L^{q^\prime}_y} dt \|f\|_{L^q} \lesssim \\
      & \int_0^\infty t^{-1/2} R( \sqrt{t+1} |x|) (1+t)^{-1/q^\prime}   dt  \|f\|_{L^q}
   \end{aligned}
   \end{equation}
   so
    \begin{equation}
    \begin{aligned}
        &\|\Gamma(f)\|_{L^{q}} \lesssim   \int_0^\infty t^{-1/2} \| R( \sqrt{t+1} |x|)\|_{L^q} (1+t)^{-1/q^\prime}   dt  \|f\|_{L^q} \\
        & \lesssim \left( \int_0^\infty t^{-1/2} (1+t)^{-1} dt \right) \ \|f\|_{L^q}
    \end{aligned}
    \end{equation}
    and we have \eqref{eq.lll82}.

    The proof of \eqref{eq.lll86} is more delicate.
We have the decomposition
$$\Gamma (f) (x) = \Gamma_{small} (f) (x) + \Gamma_{large} (f) (x), $$
where
$$\Gamma_{small} (f) (x)= \int_0^\infty t^{-1/2} R_{small}( \sqrt{t+1} |x|)  \left(\int_{\mathbb{R}^2}   \mathbb{G}_{1+t}(  |y|) \overline{f(y)}  dy\right) dt,$$
$$\Gamma_{large} (f) (x)= \int_0^\infty t^{-1/2} R_{large}( \sqrt{t+1} |x|)  \left(\int_{\mathbb{R}^2}  \mathbb{G}_{1+t}(  |y|) \overline{f(y)}  dy\right) dt $$

    As in \eqref{eq.11192} we have
  \begin{equation} \label{eq.111317}
   \left\{\begin{aligned}
     & \sum_{|\alpha|=1} |\partial_x^\alpha\Gamma_{small}(f)(x)|  \lesssim
       \int_0^\infty t^{-1/2} \sum_{|\alpha|=1} \left|\partial_x^\alpha R_{small}( \sqrt{t+1} |x|) \right| (1+t)^{-1/q^\prime}   dt  \|f\|_{L^q}, \\
        & \sum_{|\alpha|=1} |\partial_x^\alpha\Gamma_{large}(f)(x)|  \lesssim
       \int_0^\infty t^{-1/2} \sum_{|\alpha|=1} \left|\partial_x^\alpha R_{large}( \sqrt{t+1} |x|) \right| (1+t)^{-1/q^\prime}   dt  \|f\|_{L^q}.
   \end{aligned}\right.
   \end{equation}

    First we shall estimate $$ \sum_{|\alpha|=1} |\partial_x^\alpha\Gamma_{small}(f)(x)|.$$

    Using the support assumption in \eqref{eq.pss39} and the estimate $\left|\partial_r R_{small}( r )\right| \lesssim 1$ stated in \eqref{eq.pss39}, we obtain
    \begin{equation}
    \begin{aligned}
     &   \int_0^\infty t^{-1/2} \sum_{|\alpha|=1} \left| \partial_x^\alpha \left( R_{small}( \sqrt{t+1} |x|) \right)\right| \ (1+t)^{-1/q^\prime}   dt \\
     &  \lesssim  \int_0^{4/|x|^2} t^{-1/2} (1+t)^{1/2}  (1+t)^{-1/q^\prime}   dt .
    \end{aligned}
    \end{equation}
    Now we have two possibilities: $|x|$ bounded, say $|x| \leq 2$ and then the other case is $|x| >2.$
    If $|x|\leq 2,$ then
    \begin{equation}
        \begin{aligned}
        &  \int_0^{4/|x|^2} t^{-1/2} (1+t)^{1/2}  (1+t)^{-1/q^\prime}   dt  \lesssim \\
     & \lesssim  \int_0^{1/4} t^{-1/2}  dt+ \int_{1/4}^{4/|x|^2} t^{-1/2} (1+t)^{-1/2+1/q}    dt\\
     &\lesssim
     1 +   \int_{1/4}^{4/|x|^2} t^{-1+1/q}     dt \lesssim |x|^{-2/q}
        \end{aligned}
    \end{equation}
    If $|x|>2,$ then $ r= \sqrt{1+t}|x|$ is outside the support of $R_{small}(r) $ so  we can conclude
\begin{equation}
    \begin{aligned}
     &   \int_0^\infty t^{-1/2} \sum_{|\alpha|=1} \left| \partial_x^\alpha \left( R_{small}( \sqrt{t+1} |x|) \right)\right| \ (1+t)^{-1/q^\prime}   dt   \lesssim  |x|^{-2/q}
    \end{aligned}
    \end{equation}
and we arrive at
    \begin{equation}\label{eq.sme52}
        \begin{aligned}
       \sum_{|\alpha|=1} |\partial_x^\alpha\Gamma_{small}(f)(x)|  \lesssim
       |x|^{-2/q}  \|f\|_{L^q}
        \end{aligned}
    \end{equation}
Next we shall estimate $$ \sum_{|\alpha|=1} |\partial_x^\alpha\Gamma_{large}(f)(x)|.$$

For the purpose we set
 $$ Q(r) =  \sum_{|\alpha|=1} \left|\partial_x^\alpha R_{large}( |x|)\right|, r=|x|, $$

Then \eqref{eq.111317} shows that we need to evaluate
$$  \left|\int_0^\infty t^{-1/2} (1+t)^{1/2} Q( \sqrt{t+1} |x|)  \ (1+t)^{-1/q^\prime}   dt \right|.$$
 In the case $|x|\leq 1/\sqrt{2}$ we use the support assumption of $Q(r)$ and  we see that integration domain for $t$ is determined by
$$t > \frac{1}{4r^2}-1 \geq \frac{1}{8r^2}\geq 1 .$$
Hence
    \begin{equation}
    \begin{aligned}
     &   \int_0^\infty t^{-1/2} (1+t)^{1/2} Q( \sqrt{t+1} |x|)  \ (1+t)^{-1/q^\prime}   dt \\
     &  \lesssim  \int_{1/8r^2}^\infty t^{-1/2} (1+t)^{1/2} Q( \sqrt{t+1} |x|)  \ (1+t)^{-1/q^\prime}   dt .
    \end{aligned}
    \end{equation}
    We have also
    $$ \int_{1/8r^2}^\infty t^{-1/2} (1+t)^{1/2} Q( \sqrt{t+1} |x|)  \ (1+t)^{-1/q^\prime}   dt \lesssim $$
$$ \int_{1/8r^2}^\infty (1+t)^{-1+1/q} Q( \sqrt{t+1} |x|)   dt$$
Now we make change of variables $\sigma=\sqrt{t+1} |x| $ and from $$ t >  \frac{1}{8r^2} $$ we find
$$ \sigma^2 = |x|^2 + t|x|^2 \geq |x|^2 + \frac{1}8 \geq \frac{1}8 $$ so we find

$$ \int_{1/8r^2}^\infty (1+t)^{-1+1/q} Q( \sqrt{t+1} |x|)   dt \lesssim $$ $$\int_{1/2\sqrt{2}}^\infty \sigma^{-2+2/q} |x|^{2-2/q} Q( \sigma)  \   \frac{2\sigma d\sigma}{|x|^2} . $$
Applying the assumption \eqref{eq.pss47}, we see that
$$  \int_0^\infty \sigma^{-1+2/q}| \partial_\sigma R_{large}(\sigma)| d\sigma \lesssim 1$$
and we find
$$\int_{1/8r^2}^\infty (1+t)^{-1+1/q} Q( \sqrt{t+1} |x|)   dt \lesssim |x|^{-2/q}. $$

Hence we have
\begin{equation} \label{eq.lae87}
    \sum_{|\alpha|=1} |\partial_x^\alpha\Gamma_{large}(f)(x)|\lesssim |x|^{-2/q} \|f\|_{L^q}, \ \ |x| \leq 1/\sqrt{2}.
\end{equation}

We turn to the case $|x| > 1/\sqrt 2.$

Then the assumption
$$ \exists \delta>0, \ \  \mbox{so that} \  \ Q(\sigma)= |R^\prime_{large}(\sigma)| \leq e^{-\delta \sigma} , \forall \sigma  \geq 1/\sqrt2$$
from \eqref{eq.pss47} guarantees that
$$ Q( \sqrt{t+1} |x|) \leq e^{-\delta |x| \sqrt{1+t}} \lesssim e^{-\delta \sqrt{1+t}/4} e^{-\delta |x|/2}, \ \ \forall |x| >1/\sqrt2, t \geq 0.$$

Therefore, we have
$$ \sum_{|\alpha|=1} |\partial_x^\alpha\Gamma_{large}(f)(x)|  \lesssim
       \int_0^\infty t^{-1/2}  (1+t)^{-1/2+ q} \sum_{|\alpha|=1} Q( \sqrt{t+1} |x|)    dt  \|f\|_{L^q} \lesssim e^{-\delta |x|/2} \|f\|_{L^q},$$
so we arrive at
\begin{equation}\label{eq.lae02}
    \sum_{|\alpha|=1} |\partial_x^\alpha\Gamma_{large}(f)(x)|\lesssim |x|^{-2/q} \|f\|_{L^q}, \ \ |x| > 1/\sqrt2.
\end{equation}
From \eqref{eq.sme52}, \eqref{eq.lae87} and \eqref{eq.lae02} we conclude that \eqref{eq.lll86} is true.

\end{proof}

Next we consider the operator
\begin{equation}
\begin{aligned}
     & \Gamma_0 (f) (x) = \int_0^\infty t^{-1/2} [\varphi_0( \sqrt{t+1} |x|) - \varphi_0(|x|)] \left(\int_{\mathbb{R}^2}   \mathbb{G}_{1+t}( |y|) \overline{f(y)}  dy\right) dt.
\end{aligned}
\end{equation}

\begin{Lemma}\label{l.gamma0}
    We have the estimates
    \begin{equation}\label{eq.ll1423}
    \|\Gamma_0(f)\|_{L^{q}} \lesssim \|f\|_{L^q}
\end{equation}
and
\begin{equation} \label{eq.lll427}
    \sum_{|\alpha|=1}\|\partial_x^\alpha \Gamma_0(f)\|_{L^{q,\infty}} \lesssim \|f\|_{L^q}
\end{equation}
\end{Lemma}

\begin{proof}
    The proof of \eqref{eq.ll1423} is the same as the proof of \eqref{eq.lll82}. To prove \eqref{eq.lll427} we use the inequality
    $$  \sum_{|\alpha|=1} \left|  \partial_x^\alpha [\varphi_0( \sqrt{t+1} |x|) - \varphi_0(|x|)] \right| \lesssim  \left| \sqrt{1+t}\varphi_0^\prime( |x| \sqrt{1+t}) - \varphi_0^\prime(|x|) \right|,$$
    We lose no generality assuming
$$ \varphi_0(r)= - (2\pi)^{-1}\left( \log \left( r/2\right) + \gamma  \right) \phi(r),$$
where $\phi(r)$ is smooth,  $\phi(r)=1, 0 < r \leq 1$ and $\phi(r)=0, r>2.$
Then
$$ -(2\pi)\partial_r \left(  \varphi_0(r\sqrt{1+t}) - \varphi_0(r)  \right)= \frac{1}{ r}\left(\phi(r\sqrt{1+t})-\phi(r) \right)  +
$$ $$ +\left( \log \left( r/2\right) +  \log \left( \sqrt{1+t}\right) +\gamma  \right) \sqrt{1+t} \phi^\prime(r\sqrt{1+t}) - \left( \log \left( r/2\right) + \gamma  \right) \phi^\prime(r). $$

In the case, when $ \delta < r < 2$ with $\delta>0$ small we can conclude that
$$ 1+t \leq 4/r^2$$
implies $t$ is bounded, so
$$ \partial_r \left(  \varphi_0(r\sqrt{1+t}) - \varphi_0(r)  \right) = O(1), \ \delta < r < 2, \ 1+t \leq 4/r^2.$$

Further for $q>2,$ $ \delta < r < 2$ and
$1+t \geq 4/r^2 $ we have
$$  \partial_r \left(  \varphi_0(r\sqrt{1+t}) - \varphi_0(r)  \right) = \frac{1}{2\pi r}\phi(r) + O(1) $$
Now we can follow the proof of \eqref{eq.sme52} so that
$$ \sum_{|\alpha|=1}|\partial_x^\alpha\Gamma_0 (f) (x)| \lesssim r^{-1}\int_{2/r^2}^\infty t^{-1/2} (1+t)^{-1+1/q} dt\ \|f\|_{L^q} + $$ $$ +\int_0^{2/r^2} t^{-1/2} (1+t)^{-1+1/q} dt \ \|f\|_{L^q} \lesssim \|f\|_{L^q}, \ \ \delta < r < 2 $$

provided
\begin{equation}\label{eq.q265}
    q > 2.
\end{equation}

Next we turn to the case $0 < r < \delta.$
Then we can write
$$
 \frac{1}{ r}\left(\phi(r\sqrt{1+t})-\phi(r) \right) \lesssim r^{-1} \mathds{1}_{r^{-2} \leq t} . $$

Further, we have
$$\left| \left( \log \left( r/2\right) +  \log \left( \sqrt{1+t}\right) +\gamma  \right) \sqrt{1+t} \phi^\prime(r\sqrt{1+t}) \right|=$$
$$ \frac{1}{r}\left| \left( \log \left(\sqrt{1+t} r/2\right)  +\gamma  \right) r\sqrt{1+t} \phi^\prime(r\sqrt{1+t}) \right| \lesssim  r^{-1} \mathds{1}_{r^{-2} \sim t} . $$
Finally,
$$ \left|\left( \log \left( r/2\right) + \gamma  \right) \phi^\prime(r)\right| =0, $$
provided $\delta$ is small.
Hence we arrive at
$$ \sum_{|\alpha|=1}|\partial_x^\alpha\Gamma_0 (f) (x)| \lesssim r^{-1}\int_{1/(r^2)}^\infty t^{-1/2} (1+t)^{-1+1/q} dt\ \|f\|_{L^q} \lesssim r^{-2/q}\|f\|_{L^q}$$
provided $r<\delta$ and \eqref{eq.q265} holds.


\end{proof}
We finally consider the operator
\begin{equation}
\begin{aligned}
     & \Gamma_1 (f) (x) = \int_0^\infty t^{-1/2} \varphi_0( \sqrt{t+1} |x|) \left(\int_{\mathbb{R}^2}   \mathbb{G}_{1+t}( |y|) \overline{f(y)}  dy\right) dt,
\end{aligned}
\end{equation}
and we have the following Lemma.
\begin{Lemma}\label{l.gamma1}
    For any $p \in (1,2)$ the operator $\Gamma_1$ maps $L^{p}(\mathbb{R}^2)$ into $H^{1, p}(\mathbb{R}^2).$
    In particular we have the following estimates:
    \begin{equation}\label{eq.gamma1}
    \|\Gamma_1(f)\|_{L^{p}} \lesssim \|f\|_{L^p},
    \end{equation}
and
    \begin{equation}\label{eq.gamma1.2}
    \|\Gamma_1(f)\|_{H^{1,p}} \lesssim \|f\|_{L^p}.
\end{equation}
\end{Lemma}
\begin{proof}
    The proof of \eqref{eq.gamma1} is similar to \eqref{eq.lll82}. The proof of \eqref{eq.gamma1.2} follows from \eqref{eq.phi.2}.
\end{proof}

\section{Sobolev embedding and local well posedness}\label{section.5}

In this section we consider the Cauchy problem \eqref{eq.CP81}
and we shall give alternative proof of the local existence result established in Theorem B.1 in  \cite{FGI22}.

Our first step is the following Sobolev inequality.

\begin{Lemma}\label{l.sem1}
   For any $q \in (2,\infty)$ there is a constant $C=c(q)>0$ so that for any $\phi \in H^1_\alpha$ we have $\phi \in L^q$ and
   \begin{equation}
       \|\phi\|_{L^q(\mathbb{R}^2)} \leq C \|\phi\|_{H^1_\alpha}.
   \end{equation}
\end{Lemma}

\begin{proof}
    We know from \eqref{eq.defh1a81} that
    $$ \phi = g + c_* \mathbb{G}_\omega, g \in H^1.$$
    Since the classical Sobolev embedding implies
    $$ \|g\|_{L^q(\mathbb{R}^2)} \leq C \|g\|_{H^1_\alpha} $$
     moreover $G_\omega \in L^q.$ Hence,
     $$ \|\phi\|_{L^q} \lesssim \|g\|_{H^1_\alpha} + |c_*| \sim  \|\phi\|_{H^1_\alpha}.$$
\end{proof}

Further we recall the Strichartz estimates for $\Delta_\alpha$ that are obtained in \cite{CMY19,CMY19b}.

\begin{equation} \label{eq.str710}
    \begin{aligned}
& \left\| e^{i t \Delta_\alpha} f \right\|_{L^q(0,T)L^r} \lesssim \|f\|_{L^2}, \\
& \left\|\int_0^t e^{i(t-\tau)\Delta_\alpha} F(\tau) d\tau\right\|_{L^q(0,T)L^r} \lesssim  \left\| F\right\|_{L^{\tilde{q}^\prime}(0,T)L^
{\tilde{r}^\prime}},
    \end{aligned}
\end{equation}
provided
\begin{equation}\label{eq.adm18}
    \frac{1}{q} + \frac{1}{r} = \frac{1}{2}, \ q \in (2, \infty], \ \frac{1}{\tilde{q}} + \frac{1}{\tilde{r}} = \frac{1}{2}, \ \tilde{q} \in (2, \infty]
\end{equation}

It is easy to obtain local well - posedness of the problem \eqref{eq.CP81} in the mass subcritical case.

\begin{proof}[Proof of Theorem \ref{t.le19}]
 Consider the operator
 \begin{equation}\label{eq.kop37}
     \mathfrak{K} (u) = e^{it \Delta_\alpha} u_0 -i \int_0^t e^{i(t-\tau)\Delta_\alpha}  u(\tau)|u(\tau)|^{p-1} d\tau
 \end{equation}
    and define the Banach space $$L^\infty(0,T)L^2 \cap L^{\tilde{q}^\prime}(0,T) L^{\tilde{r}^\prime}$$ with

Applying the Strichartz estimate with
\begin{equation}\label{tqr46}
  \tilde{r} =2 ,\ \tilde{q}= \infty ,
\end{equation}
we get
  $$ \left\| \mathfrak{K}(u)\right\|_{L^q(0,T)L^r} \lesssim \|u_0\|_{L^2} + \left\| u|u|^{p-1}\right\|_{L^{1}(0,T)L^{2}  }  \lesssim \|u_0\|_{L^2} + \left\|u\right\|^p_{L^{p}(0,T)L^{2p}  } . $$
  so we can choose
  \begin{equation}
      \begin{aligned}
          & r=2p, \\
          & q = \frac{2p}{p-1}.
      \end{aligned}
  \end{equation}
so that $(q,r)$ is admissible couple satisfying \eqref{eq.adm18}. Now we need
 $$ \left\|u\right\|_{L^{p}(0,T)L^{2p}  }  \lesssim T^{\alpha}  \left\|u\right\|_{L^{q}(0,T)L^{2p}  }  $$
 with $\alpha=(3-p)/(2p)>0$ and this can be done if and only if
 $$ p < q = \frac{2p}{p-1} $$
 that is $p<3.$ The estimate
 \begin{equation}\label{eq.sest60}
     \left\| \mathfrak{K}(u)\right\|_{L^{2p/(p-1)}(0,T)L^{2p}} \lesssim \|u_0\|_{L^2} + T^{p\alpha} \left\| \mathfrak{K}(u)\right\|^p_{L^{2p/(p-1)}(0,T)L^{2p}}
 \end{equation}
    shows that $\mathfrak{K} $ maps
  \begin{equation}\label{eq.2R62}
     \left\{ u \in L^{2p/(p-1)}(0,T)L^{2p}; \left\| u\right\|_{L^{2p/(p-1)}(0,T)L^{2p}} \leq 2R \right\}
  \end{equation}
  provided $u_0 \in B_{L^2}(R)$
and $T=T(R,p)$ is sufficiently small.
In a similar way we deduce
$$ \left\| \mathfrak{K}(u)-\mathfrak{K}(\tilde{u})\right\|_{L^{2p/(p-1)}(0,T)L^{2p}} \lesssim \frac{1}{2} \left\| u-\tilde{u}\right\|_{L^{2p/(p-1)}(0,T)L^{2p}}  $$
so $\mathfrak{K}$ is a contraction in \eqref{eq.2R62}.

Observing that the estimate \eqref{eq.sest60} and Strichartz estimates imply
\begin{equation}\label{eq.sest75}
     \left\| \mathfrak{K}(u)\right\|_{L^{q}(0,T)L^{r}} \lesssim \|u_0\|_{L^2} + T^{p\alpha} \left\| u\right\|^p_{L^{2p/(p-1)}(0,T)L^{2p}}
 \end{equation}
 for any admissible couple, we complete the proof.

\end{proof}

Our Theorem \ref{t.2.1} guarantees the more general Strichartz estimates
\begin{equation} \label{eq.SH123}
    \begin{aligned}
& \left\| e^{i t \Delta_\alpha} f \right\|_{L^q(0,T)H^{1,r}_\alpha} \lesssim \|f\|_{H^1_\alpha}, \\
& \left\|\int_0^t e^{i(t-\tau)\Delta_\alpha} F(\tau) d\tau\right\|_{L^q(0,T)H^{1,r}_\alpha} \lesssim  \left\| F\right\|_{L^{\tilde{q}^\prime}(0,T)H^{1,{\tilde{r}^\prime}}_\alpha  }
    \end{aligned}
\end{equation}

Our next local existence result treats the case $p \geq 3$.
\begin{Theorem}\label{t.lems8}
    For any $p\geq 3$ and any $R>0$ there  exists $T=T(R,p)>0$ so that for any
    $$ u_0 \in B(R) = \left\{ \phi \in H^1_\alpha; \|\phi\|_{H^1_\alpha} \leq R \right\}$$
    there exists a unique solution
    $$ u \in C([0,T]; H^1_\alpha)$$
    to the integral equation
\begin{equation}
    u = e^{it \Delta_\alpha} u_0 -i \int_0^t e^{i(t-\tau)\Delta_\alpha}  u(\tau)|u(\tau)|^{p-1} d\tau.
\end{equation}
    associated to \eqref{eq.CP81}.
\end{Theorem}
\begin{proof}
    Consider the operator
  $$  \mathfrak{K} (u) = e^{it \Delta_\alpha} u_0 -i \int_0^t e^{i(t-\tau)\Delta_\alpha}  u(\tau)|u(\tau)|^{p-1} d\tau. $$

   Further, we define  the Banach space $\mathcal{B}=L^\infty(0,T)H^1_\alpha $ and the corresponding ball of radius $R$
  $$ B_{\mathcal{B}} = \left\{u \in \mathcal{B} ; \|u\|_{\mathcal{B}} \leq R \right\}.$$

  Applying the Strichartz estimate \eqref{eq.SH123}, we find
   $$ \left\| \mathfrak{K}(u)\right\|_{L^q(0,T)H^1_\alpha} \lesssim \|u_0\|_{H^1_\alpha} + \left\| u|u|^{p-1}\right\|_{L^{\tilde{q}^\prime}(0,T)H^{1, \tilde{r}^\prime}_\alpha  } .  $$

  Now we choose

  \begin{equation}\label{eq.rq}
  \tilde{r} = \frac{2-\varepsilon}{1-\varepsilon},\ \tilde{q}= \frac{4-2\varepsilon}{\varepsilon}
\end{equation}
  so that
  \begin{equation}\label{eq.rqprime}
  \tilde{r}^\prime = 2-\varepsilon,\ \tilde{q}^\prime= \frac{\varepsilon}{3\varepsilon-4}  .
\end{equation}
Since $\tilde{r}^\prime <2,$ we see that  Theorem \ref{t.2.1} implies
$$ \left\| u|u|^{p-1}\right\|_{H^{1,{\tilde{r}^\prime}}_\alpha  } \sim \left\| u|u|^{p-1}\right\|_{H^{1,{\tilde{r}^\prime}}  } \sim
\left\| u|u|^{p-1}\right\|_{L^{\tilde{r}^\prime} } + \left\| \nabla u|u|^{p-1}\right\|_{L^{\tilde{r}^\prime} }$$
Now we use the fact that $u \in H^1_\alpha$
$$  u= g + c_*G_\omega$$
and we can continue the estimates  as follows

$$ \left\| |\nabla u||u|^{p-1}\right\|_{L^{\tilde{r}^\prime} } \leq \left\| |\nabla g||u|^{p-1}\right\|_{L^{\tilde{r}^\prime} } +
|c_*|\left\| |\nabla \mathbb{G}_\omega||u|^{p-1}\right\|_{L^{\tilde{r}^\prime} }.$$
Then we estimate each of the terms in the right side and find
$$\left\| |\nabla g||u|^{p-1}\right\|_{L^{\tilde{r}^\prime} }\lesssim\|\nabla  g\|_{L^2}\|u\|^{(p-1)}_{L^{\frac{2(2-\varepsilon)}{\varepsilon}(p-1)}}\lesssim \|\nabla  g\|_{L^2}\|u\|^{p-1}_{H^1_\alpha},$$

$$\left\| |\nabla \mathbb{G}_\omega||u|^{p-1}\right\|_{L^{\tilde{r}^\prime} }\lesssim\|\nabla \mathbb{G}_\omega\|_{L^{2-\frac{\varepsilon}{2}}}\|u\|_{L^{\frac{(2-\varepsilon)(4-\varepsilon)}{\varepsilon}(p-1)}}^{(p-1)}\lesssim \|u\|^{p-1}_{H^1_\alpha}.$$

Hence
\begin{equation}
 \left\| |\nabla u||u|^{p-1}\right\|_{L^{\tilde{r}^\prime} } \lesssim    (\left\| \|\nabla  g\|_{L^2}+ |c_*| \right) \|u\|^{p-1}_{H^1_\alpha}
\end{equation}
and via \eqref{eq.ns77} we get
\begin{equation}
 \left\| |\nabla u||u|^{p-1}\right\|_{L^{\tilde{r}^\prime} } \lesssim     \|u\|^{p}_{H^1_\alpha}.
\end{equation}

So

$$ \left\| \mathfrak{K}(u)\right\|_{L^q(0,T)H^{1,r}_\alpha} \lesssim \|u_0\|_{H^1_\alpha} +  T^{p/\tilde{q}'}\|u\|^{p}_{L^\infty(0,T) H^1_\alpha}$$

In a similar way we deduce
$$ \left\| \mathfrak{K}(u)-\mathfrak{K}(\tilde{u})\right\|_{L^{\infty}(0,T) H^1_\alpha} \lesssim \frac{1}{2} \left\| u-\tilde{u}\right\|_{L^\infty(0,T) H^1_\alpha}  $$
so $\mathfrak{K}$ is a contraction in $L^\infty(0,T)H^1_\alpha. $

\end{proof}
\appendix
\section{Rescaling}

\label{sec.res2}

We recall the rescaling argument from section  5 in \cite{FGI22}.
Let
\begin{equation}
    \phi(x) = g(x) + \frac{g(0)}{\beta_\alpha(\omega)} \mathbb{G}_\omega(x) \in \mathcal D(\Delta_\alpha)
\end{equation}
so that
\begin{equation}
    (\omega -\Delta_\alpha) \phi = \Phi
\end{equation}
and
\begin{equation}
    \begin{aligned}
      &  \tilde{\phi} (x)  = \phi \left( \frac{x}{\sqrt{\omega}}\right), \\
       &  \tilde{\Phi} (x)  = \Phi \left( \frac{x}{\sqrt{\omega}}\right), \\
      &\tilde{g} (x)  = g \left( \frac{x}{\sqrt{\omega}}\right),\\
      & \tilde{\alpha} = \alpha + \frac{1}{4\pi} \ln (\omega).
    \end{aligned}
\end{equation}
It is easy to deduce
\begin{equation}
    \begin{aligned}
       & \beta_\alpha(\omega) = \beta_{\tilde{\alpha}}(1), \\
       & \tilde{\phi} (x) = \tilde{g} (x)  +   \frac{\tilde{g}(0)}{\beta_{\tilde{\alpha}}(1)} \mathbb{G}_1(x)
    \end{aligned}
\end{equation}
and moreover
\begin{equation}\label{eq.rescalingDelta}
    \widetilde{(\omega-\Delta_\alpha) \phi} = \omega (1-\Delta_{\tilde{\alpha}}) \tilde{\phi}.
\end{equation}
Applying the spectral theorem we find
\begin{equation}\label{eq.rescalingDelta0.5}
    \widetilde{(\omega-\Delta_\alpha)^{s/2}\phi} = \omega^{s/2} (1-\Delta_{\tilde{\alpha}})^{s/2} \tilde{\phi}, \ \forall s \in [0,2].
\end{equation}

We turn to the rescaling of the   linear Schr\"odinger  equation:
$$(i\partial_t+\Delta_\alpha)u=0.$$
With the change of variable
$$y=\frac{x}{\sqrt{\omega}}, \ \ s=\frac{t}{\omega},$$
it is easy to see that
$$u\left(\frac{x}{\sqrt{\omega}},\frac{t}{\omega}\right)=e^{it\Delta_{\tilde\alpha}}\tilde{u_0}.$$
This change, works also for the NLS
$$(i\partial_t+\Delta_\alpha)u=   \mu u|u|^{p-1}.$$

We consider the rescaling $S_\lambda(u)(t,x) = \lambda^{-1} u \left(\frac{t}{\lambda^2}, \frac{x}{\lambda} \right)$
we have that
$$    \|S_\lambda(u)(t)\|^2_{L^{2}(\mathbb{R}^2)} =    \|u(t)\|^2_{L^{2}(\mathbb{R}^2)}$$
and
$$\lambda^{-1}(\omega-\Delta_\alpha)u(\frac{t}{\lambda^2},\frac{x}{\lambda})=\lambda(\omega-\Delta_{\alpha+\frac{1}{2\pi}\ln(\lambda)})(u(\frac{t}{\lambda^2},\frac{x}{\lambda})),$$

that gives

\begin{equation}
    \|S_\lambda(u)(t)\|^2_{L^2(\mathbb{R}^2)} = \|u(t)\|^2_{L^2(\mathbb{R}^2)}
\end{equation}
and
\begin{equation}
\begin{aligned}
     E(S_\lambda(u))=&\frac{1}{2}\int_{\mathbb{R}^2} (\omega-\Delta_\alpha) S_\lambda(u)(t) \overline{S_\lambda(u)(t)} -\frac{\omega}{2}\|S_\lambda (u)(t)\|^{2}_{L^{2}(\mathbb{R}^2)}+ \frac{\mu}{p+1} \|S_\lambda (u)(t)\|^{p+1}_{L^{p+1}(\mathbb{R}^2)}=\\
     &=\lambda^{-2}\frac{1}{2}\|(\omega-\Delta_{\alpha^*})^{1/2}u(t)\|_{L^2}^2-\lambda^{-2}\frac{\omega}{2}\|u(t)\|^{2}_{L^{2}(\mathbb{R}^2)}+\lambda^{1-p}\frac{\mu}{p+1} \| u(t)\|^{p+1}_{L^{p+1}(\mathbb{R}^2)},
\end{aligned}
\end{equation}
where $\alpha^*=\alpha+\frac{1}{2\pi}\ln(\lambda)$. Note that $\alpha^* > \alpha$ if $\lambda>1.$

Now it is clear that the mass  critical case is defined by  $\lambda^{-2} = \lambda^{1-p}$ with $\lambda >1$ so
mass critical case is $p=3.$

\section{Asymptotics}
\label{sec.as8}

The function $K_0 ( \sqrt{z}|x|) $ has the asymptotics
\begin{equation}
    K_0( \sqrt{z}r) = - \log \left(r\right) - c(\omega) + O(|\sqrt{z}|r), \ \  \ c(\omega)= \log\left(\frac{\sqrt{\omega}}{2} \right)+\gamma
\end{equation}
where $$r|\sqrt{z}| \leq 1, \ \ z \in \mathbb{C}, \mathrm{arg}(z) \in (-\pi, \pi). $$
This follows from the following asymptotic expansion (see (38), p.9 in \cite{BE})
\begin{equation}
    \begin{aligned}
      & K_0(z)=-I_0(z) \log \left(\frac{z}{2}\right)+ \sum_{m=0}^{\infty}  \left(\frac{z}{2}\right)^{2 m} \frac{\psi(m+1)}{ \left[(m !)^2\right]}= \\
      = & -  \log \left(\frac{z}{2}\right) -\gamma + O(\log(1/|z|) |z|^2), \ \ |z|\leq 1,
    \end{aligned}
\end{equation}
where $\psi(z)=\Gamma^\prime(z)/\Gamma(z)$ and $\gamma$ is the Euler-Mascheroni constant.
We have also
\begin{equation}
    \begin{aligned}
      & K_0^\prime(z)=-\frac{2}{z} + O(\log(1/|z|) |z|), \ \ |z|\leq 1, \\
      & K_0^{\prime\prime}(z)=\frac{2}{z} + O(\log(1/|z|) ), \ \ |z|\leq 1,
    \end{aligned}
\end{equation}

We have the following asymptotic expansion valid if $|z| \to \infty$ and $\mathrm{arg} z \in (-\pi, \pi)$ (see relation (20), section 7.23 in \cite{W95})
\begin{equation}\label{eq.Bf7a23}
    K_\nu(z) = \left( \frac{\pi}{2} \right)^{1/2} e^{- (\log|z| +\mathrm{i}\mathrm{arg} z)/2} e^{-z} \left( 1+ O(|z|^{-1}) \right).
\end{equation}

\bibliographystyle{habbrv}
 \bibliography{SING_PERT_GR}
\end{document}